\documentclass{amsart}
\usepackage{amsmath,amsthm,amsfonts,amssymb}
\usepackage{latexsym}
\usepackage{mathrsfs}
\usepackage[all]{xy}
\usepackage{url}
\usepackage[pdftex]{graphicx}
\usepackage{float}
\usepackage{tikz}
\usepackage{longtable}
\usetikzlibrary{decorations.markings} 
\usepackage{hyperref}
\usepackage{ upgreek }
\usepackage{amssymb}
\usepackage{amsmath}
\usepackage{tikz}
\usepackage{comment}
\usetikzlibrary{decorations.markings} 
\usepackage[shortlabels]{enumitem}

\makeatletter
\newtheorem*{rep@theorem}{\rep@title}
\newcommand{\newreptheorem}[2]{\newtheorem*{rep@#1}{\rep@title}
\newenvironment{rep#1}[1]{\def\rep@title{#2 \ref*{##1}}
\begin{rep@#1}}
{\end{rep@#1}}} 
\makeatother

\theoremstyle{plain}
  \newtheorem{thm}{Theorem}[section]
  \newreptheorem{thm}{Theorem}
  \newtheorem*{thm*}{Theorem}
  
  \newreptheorem{prop}{Proposition}
  \newtheorem{lem}[thm]{Lemma}
  \newtheorem{cor}[thm]{Corollary}
   \newreptheorem{cor}{Corollary}
  
\theoremstyle{definition}

\theoremstyle{remark}
  \newtheorem{remark}[thm]{Remark}
\numberwithin{equation}{section}

\makeatletter
\@namedef{subjclassname@2020}{%
  \textup{2020} Mathematics Subject Classification}
\makeatother

\begin{document}

\title[A $2$-complex with contracting non-positive immersions]{A $2$-complex with contracting non-positive immersions and positive maximal irreducible curvature}

\author{Mart\'in Axel Blufstein and El\'ias Gabriel Minian}

\address{Departamento  de Matem\'atica - IMAS\\
 FCEyN, Universidad de Buenos Aires. Buenos Aires, Argentina.}
\email{mblufstein@dm.uba.ar}
\email{gminian@dm.uba.ar}

\thanks{Researchers of CONICET. Partially supported by grants PICT 2019-2338 and UBACYT 20020190100099BA}



\begin{abstract}
We prove that the $2$-complex associated to the presentation $\langle a,b \mid b,bab^{-1}a^{-2}\rangle$ has contracting non-positive immersions and  positive maximal irreducible curvature. This example shows that the contracting non-positive immersions property is not equivalent to the notion of non-positive irreducible curvature, answering a question raised by H. Wilton.
\end{abstract}

\maketitle


\section{Introduction}

In \cite{Wis03,Wis04} D. Wise introduced the notions of sectional curvature and non-positive immersions for $2$-complexes in order to study properties of their fundamental groups such as coherence and local indicability. In \cite{Wis22} he investigated some variations of the non-positive immersions property.  Here we are concerned with the contracting variant. Recall that an immersion $Y\looparrowright X$ is a map between $2$-complexes that is locally injective. A $2$-complex $X$ has \emph{contracting non-positive immersions} if for any immersion  $Y\looparrowright X$ with $Y$ connected and compact, either $\chi(Y)\leq0$ or $Y$ is contractible. Here  $\chi(Y)$ denotes the Euler characteristic of $Y$. One of the main tools used in \cite{Wis03,Wis04,Wis22} is the theory of towers, due to J. Howie \cite{How81}. Wise remarked that the proofs of the non-positive immersions property for some $2$-complexes are rather ad-hoc and asked whether there exists an algorithm to recognize if a compact $2$-complex has this property \cite[Problem 1.9]{Wis22}.

More recently H. Wilton \cite{Wil22a} defined new curvature invariants of $2$-complexes. These invariants are related to Wise's previous notions although Wilton's approach and techniques are different. The \emph{average curvature} $\kappa(X)$ of a finite $2$-complex $X$ is defined as $\kappa(X)=\frac{\chi(X)}{Area(X)}$ where $Area(X)$ is the number of $2$-cells, and the \emph{maximal irreducible curvature} $\rho_+(X)$ is the supremum of the curvatures $\kappa(Y)$ among all finite irreducible branched $2$-complexes $Y$ admitting essential maps $Y\to X$ (see \cite{Wil22a} for more details). In  \cite{Wil22b} it is proved that, when $X$ is irreducible, this supremum is attained  and  is algorithmically computable. Wilton showed that if $X$ has non-positive irreducible curvature (i.e.  $\rho_+(X)\leq 0$), then it has contracting non-positive immersions, and asked whether there is an example of a $2$-complex with contracting non-positive immersions with $\rho_+(X)>0$.

We give here an example of such a $2$-complex. The presentation $P = \langle a,b \mid b, bab^{-1}a^{-2}\rangle$ belongs to a family of balanced presentations of the trivial group introduced by Miller and Schupp \cite{MS99}. It is clear that its associated $2$-complex $K_P$ is contractible. This example appeared in Wilton's paper \cite[Example 4.8]{Wil22a} where it is proved that $K_P$ is irreducible and therefore $\rho_+(K_P)\geq \kappa(K_P)=\frac12$. In \cite{Fis22}, W. Fisher investigated Miller and Schupp's family of examples and showed that some of them fail to have the non-positive immersions property. However, his methods could not conclude for  $P = \langle a,b \mid b, bab^{-1}a^{-2}\rangle$. We prove the following.

\begin{thm}\label{main}
	The $2$-complex $K_P$ associated to the presentation $\langle a,b \mid b, bab^{-1}a^{-2}\rangle$ has contracting non-positive immersions.
\end{thm}

Our approach uses techniques developed by Wilton in \cite{Wil22a} and by Louder and Wilton in \cite{LW20}, where the immersions are built by performing foldings to the $2$-complexes.


\section{The Proof}
We work in the category of combinatorial $2$-complexes and combinatorial maps (which send (open) $n$-cells homeomorphically to $n$-cells). We will use the notion of folding of $2$-complexes introduced by Louder and Wilton in \cite{LW20}, which is a natural generalization of Stallings' foldings \cite{Sta83}.
Given a combinatorial map between  $2$-complexes $X\to Y$, we first fold the $1$-skeleton of $X$ and then glue the $2$-cells accordingly to obtain an immersion $Z\looparrowright Y$.

Given an immersion $X\looparrowright K_P$, the $2$-cells of $X$ which are mapped to the cell corresponding to the relator $b$ are called of \emph{type 1}, and the $2$-cells mapped to the cell corresponding to $bab^{-1}a^{-2}$ are of \emph{type 2}. 
By the work of Helfer and Wise \cite{HW16} and, independently, by Louder and Wilton \cite{LW17}, it is known that torsion-free one-relator groups have contracting non-positive immersions.
Thus, if an immersion $X\looparrowright K_P$ has only  $2$-cells of type $1$ or of type $2$, then $\chi(X)\leq 0$ or $X$ is contractible.
Therefore, we can reduce our study to immersions that use both types of $2$-cells.
We can also consider only immersions without free faces, since collapsing a free face does not change the homotopy type.
We will show that any such immersion is actually contractible.

In what follows, \emph{coupling} a $2$-cell to an immersion $X\looparrowright K_P$ will mean the immersion obtained by gluing a closed $2$-cell along an edge to $X$ and folding. From now on, the immersions $X\looparrowright K_P$ will be denoted with single letters (see the definitions of immersions $D_i$ and $C_i$ below). The immersion will be specified by giving orientations and  labelings to the $1$-skeleton of the domain.

We begin by defining two families of immersions $D_i$ and $C_i$. Start with the immersion of a single $2$-cell of type $1$.
We can couple a $2$-cell of type $2$ to this single $2$-cell along the first $b$ of $bab^{-1}a^{-2}$, or the second $b$ (which appears as $b^{-1}$).
We will couple it along the second appearance (the following analysis is completely analogous if we decide to start with the first appearance).
We get a new disc that also has a free face labeled by $b$ (see Figure \ref{figone}).

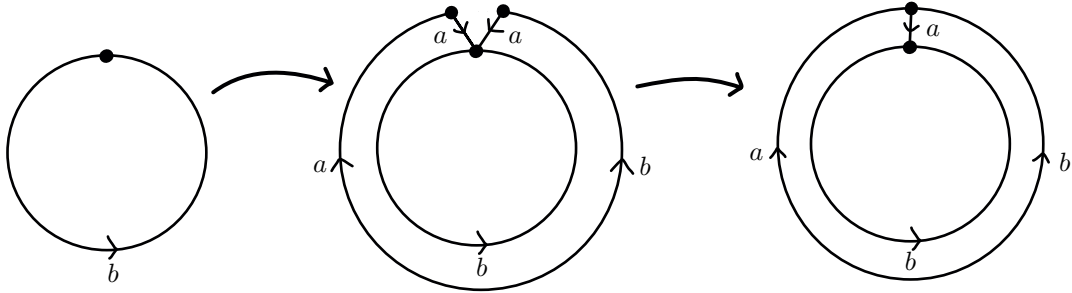
\begin{figure}[ht]
\begin{tikzpicture}
\node at (0,0) {\includegraphics[scale=0.5]{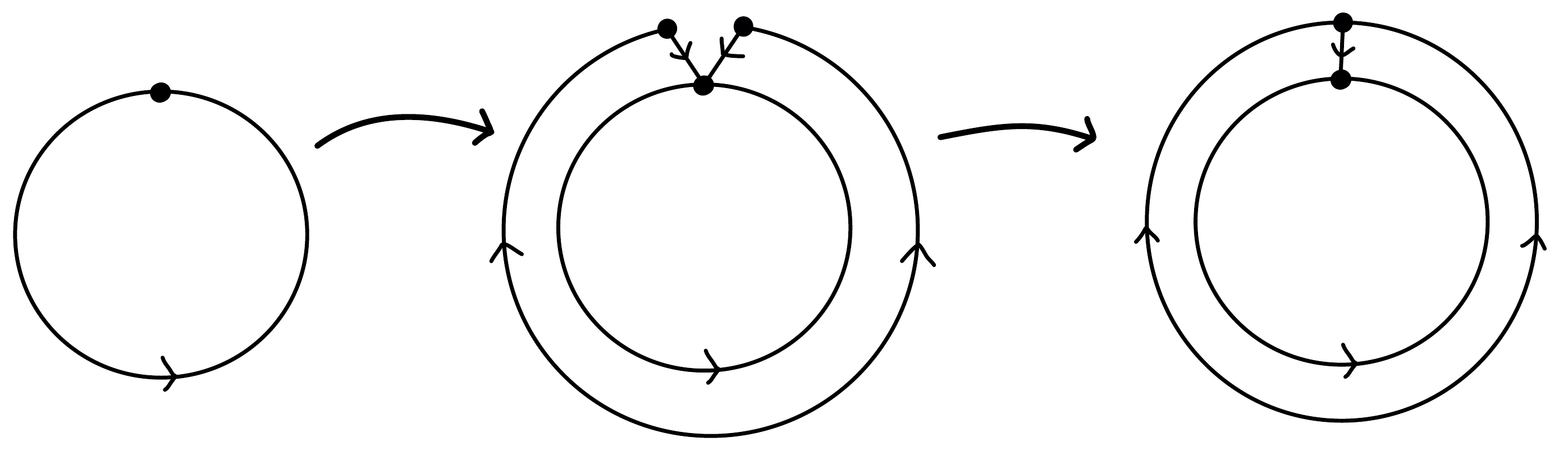}};

\node at (-5.55,-1.68) {$b$};

\node at (-0.65,-1.6) {$b$};
\node at (1.52,-0.26) {$b$};
\node at (-2.8,-0.26) {$a$};
\node at (-0.2,1.45) {$a$};
\node at (-1.2,1.45) {$a$};

\node at (7.1,-0.2) {$b$};
\node at (5.05,-1.52) {$b$};
\node at (3,-0.1) {$a$};
\node at (5.35,1.55) {$a$};
\end{tikzpicture}
\caption{Coupling a $2$-cell of type $2$ to a $2$-cell of type $1$.}
\label{figone}
\end{figure}

We can once again couple a $2$-cell of type $2$ to this immersion along the second appearance of $b$.
Again, we arrive to a disc with a free face labeled by $b$.
We can continue with this process and obtain diagrams $D_0,D_1,D_2,\ldots$ as in Figure \ref{figtwo}.

\begin{figure}[ht]
\begin{tikzpicture}
\node at (0,0) {\includegraphics[scale=0.5]{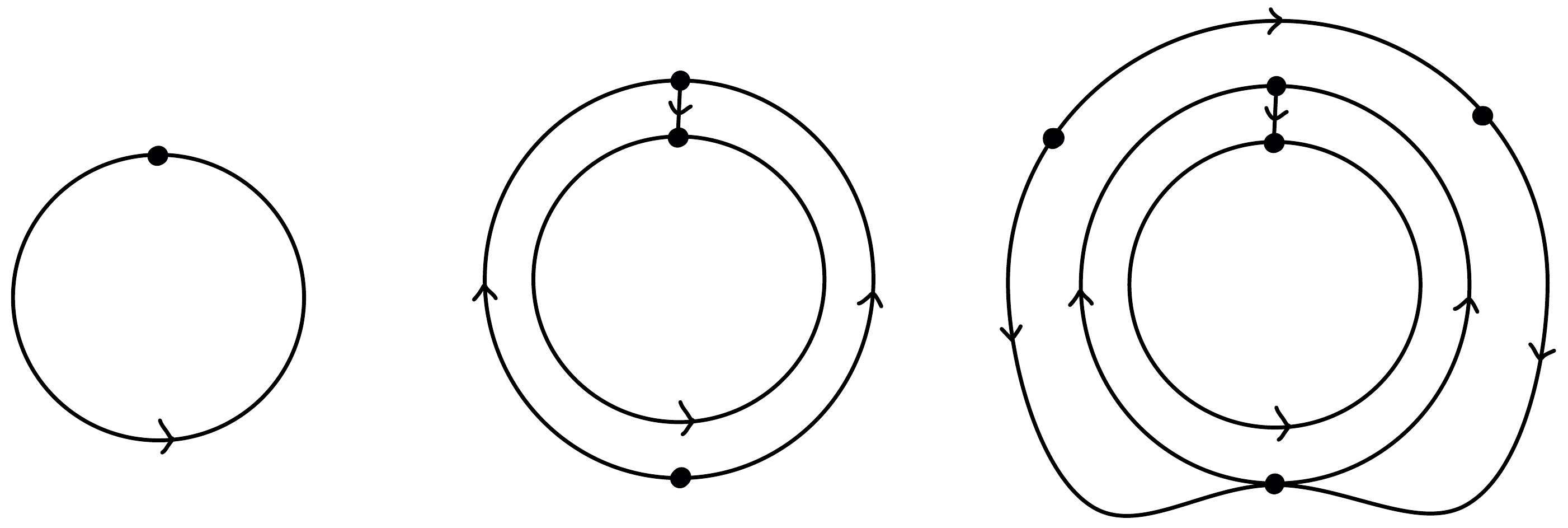}};

\node at (-6.9,1.8) {\Large{$D_0$}};
\node at (-2.6,2.3) {\Large{$D_1$}};
\node at (2.5,2.7) {\Large{$D_2$}};

\node at (-5.7,1.25) {$v_0$};
\node at (-5.6,-1.9) {$b_0$};

\node at (-1,0.85) {$v_0$};
\node at (-0.93,1.92) {$v_1$};
\node at (-1,-2.2) {$v_2$};
\node at (-1,-1.65) {$b_0$};
\node at (1.1,-0.3) {$b_1$};
\node at (-0.6,1.35) {$a_1$};
\node at (-3.05,-0.3) {$a_2$};

\node at (4.4,0.8) {$v_0$};
\node at (4.45,1.85) {$v_1$};
\node at (4.45,-2.25) {$v_2$};
\node at (6.65,1.5) {$v_3$};
\node at (2.15,1.3) {$v_4$};
\node at (4.4,-1.7) {$b_0$};
\node at (6.55,-0.39) {$b_1$};
\node at (1.8,-0.8) {$b_2$};
\node at (4.8,1.35) {$a_1$};
\node at (2.3,-0.3) {$a_2$};
\node at (7.2,-0.9) {$a_3$};
\node at (4.4,2.5) {$a_4$};
\end{tikzpicture}
\caption{Discs $D_0$, $D_1$ and $D_2$ with their labelings.}
\label{figtwo}
\end{figure}
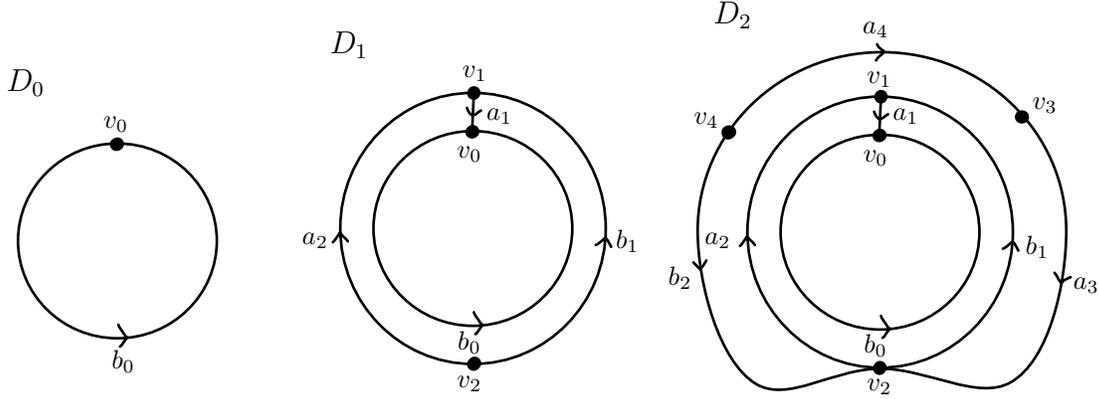

We number the labels of the edges of the $D_i$.
The $b$ in $D_0$ is labeled by $b_0$.
At each step we are adding two $a$'s and one $b$.
The $b$ added in step $i$ is labeled by $b_i$, and the $a$'s by $a_{2i-1}$ and $a_{2i}$ as shown in Figure \ref{figtwo}.
In addition, we label the vertices so that the edge labeled by $a_i$ goes from $v_i$ to $v_{i-1}$.
With this numbering, we can give a more explicit description of the $1$-skeleton of $D_i$. It has:
\begin{itemize}
    \item $2i+1$ vertices $v_0,\ldots,v_{2i}$;
    \item $2i$ oriented edges $a_1,\ldots,a_{2i}$, where $a_j$ goes from $v_j$ to $v_{j-1}$;
    \item $i+1$ oriented edges $b_0,\ldots,b_i$, where $b_j$ goes from $v_{2j}$ to $v_j$.
\end{itemize}

Now we define the immersions $C_i$ for $i\geq 1$.
The $2$-complex $C_i$ is obtained from $D_i$ by identifying $b_i$ with $b_0$. Consider first the case $i=1$. Note that in that case we must identify also $a_1$ with $a_2$ in order to obtain an immersion (see Figure \ref{figthree}). With these identifications we get $C_1=K_P$, which is contractible. In general, when $i$ is odd, we must identify $a_{2i-k}$ with $a_{i-k}$ for all $0\leq k \leq i-1$.
This corresponds to identifying the exterior stretch of the path of $a$'s to the interior one.
Therefore, the $2$-complex $C_i$ is homeomorphic to $K_P$ and, in particular, it is contractible. Note that the procedure of identifications that we just described is analogous to the one for $D_1$ (see Figure \ref{figthree}). We give an explicit description of the $1$-skeleton of $C_i$ when $i$ is odd. It has:
\begin{itemize}
    \item $i$ vertices $v_0,\ldots,v_{i-1}$;
    \item $i$ oriented edges $a_1,\ldots,a_i$, where $a_j$ goes from $v_j$ to $v_{j-1}$, and $a_i$ goes from $v_0$ to $v_{i-1}$;
    \item $i$ oriented edges $b_0,\ldots,b_{i-1}$ where $b_j$ goes from $v_{(2j \mod i)}$ to $v_j$.
\end{itemize}

\begin{figure}[ht]
\begin{tikzpicture}
\node at (0,0) {\includegraphics[scale=0.4]{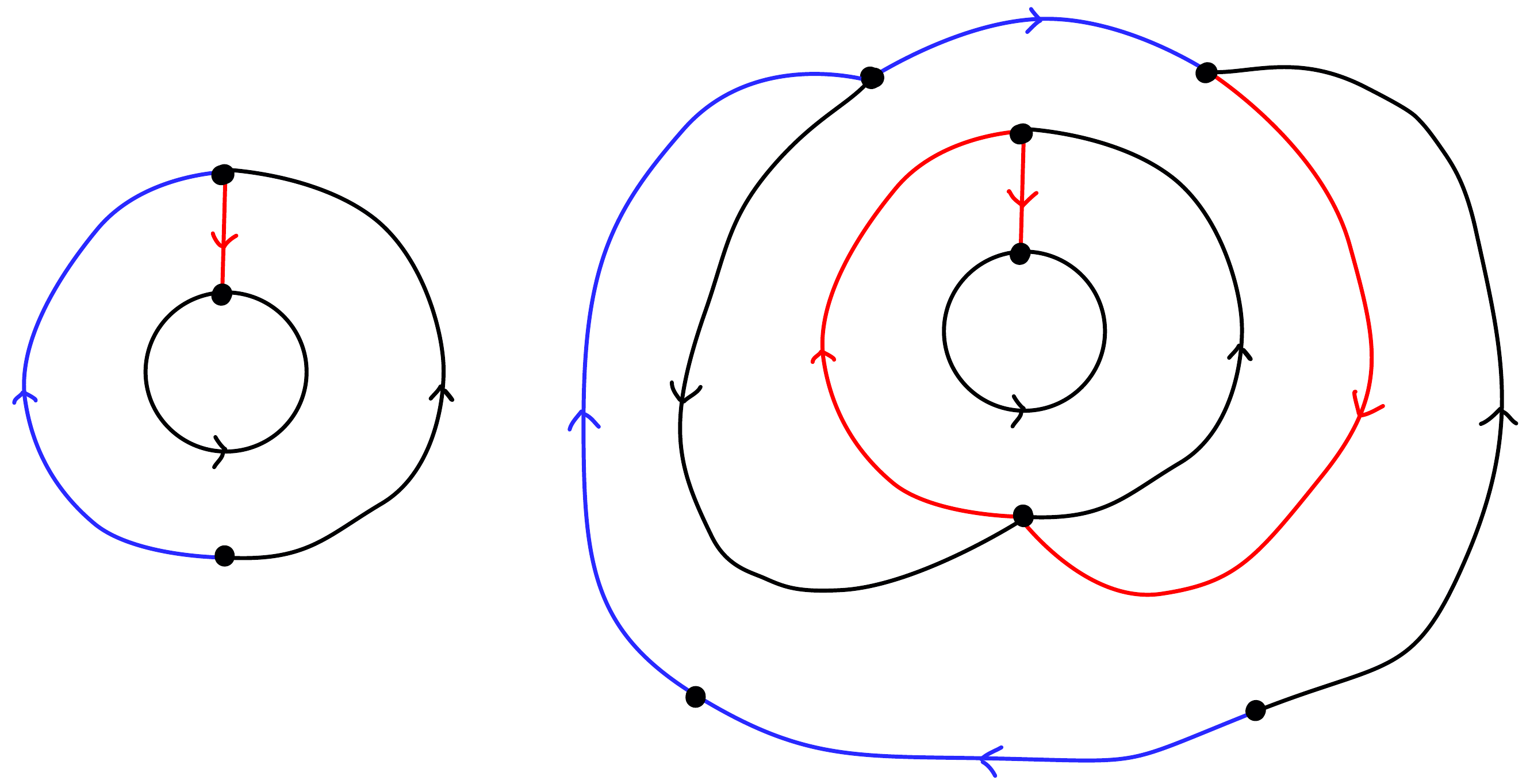}};

\node at (-5.3,2) {\Large{$D_0$}};
\node at (-0.7,2.9) {\Large{$D_3$}};

\node at (-4.15,1.05) {\textcolor{red}{$a_1$}};
\node at (-5.55,-0.05) {\textcolor{blue}{$a_2$}};

\node at (1.5,1.3) {\textcolor{red}{$a_1$}};
\node at (0.1,0.25) {\textcolor{red}{$a_2$}};
\node at (4.6,-0.2) {\textcolor{red}{$a_3$}};
\node at (2,2.85) {\textcolor{blue}{$a_4$}};
\node at (-1.56,-0.3) {\textcolor{blue}{$a_5$}};
\node at (1.7,-2.9) {\textcolor{blue}{$a_6$}};
\end{tikzpicture}
\caption{We identify the blue and red paths in $D_1$ and $D_3$.}
\label{figthree}
\end{figure}
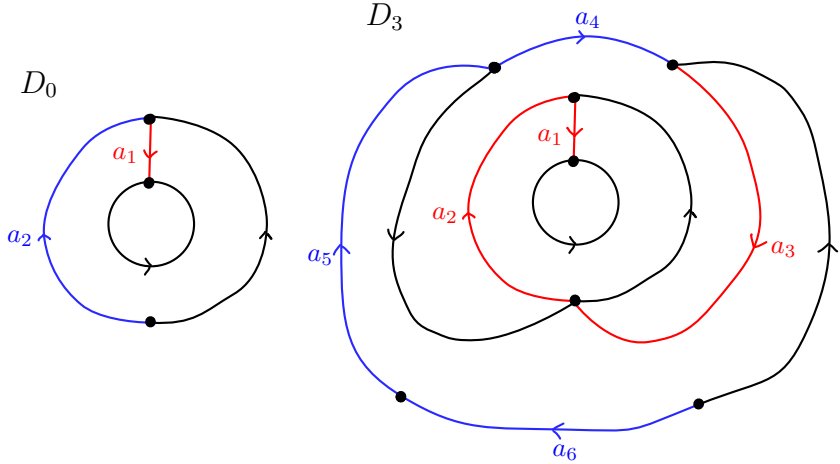

\begin{remark}\label{2saturated}
Every edge labeled by $a$ in $C_i$ is a face of three $2$-cells of type $2$ (counted with  multiplicity).
Also, every edge labeled by $b$ is adjacent to two $2$-cells of type $2$ (again, counted with  multiplicity). This implies that every vertex has degree 4 and it is adjacent to two edges labeled with the letter $a$ (with opposite orientations) and two edges labeled with the letter $b$ (with opposite orientations).
\end{remark}

We now deal with the case where $i$ is even.
We denote by $odd(i)$ the largest odd divisor of $i$.
After identifying $b_i$ with $b_0$, we must also identify $b_i$ with $b_{i/2}$, $b_{i/2}$ with $b_{i/4}$, and so on, until we get to $b_{odd(i)}$.
Then by following the path of $a$'s we also have to identify $v_j$ with $v_k$ if $j\equiv k \mod odd(i)$.
Therefore, we get that $C_i=C_{odd(i)}$.

If we start the process of constructing the immersions by coupling the first $2$-cells of type $2$ along the other $b$ (the first appearance of $b$ in $bab^{-1}a^{-2}$), we obtain immersions $\tilde{D}_i$ and $\tilde{C}_i$, where the only difference with the previous ones is that the orientations of all the $a$'s are reversed.
The preceding observations and all of the following results hold for these immersions in a completely analogous manner.

Our goal now is to show that any immersion over $K_P$ that uses $2$-cells of both types and that does not have free faces must be a $C_i$ or a $\tilde{C}_i$ for some $i$, and therefore contractible.
For that we need the following lemmas.

\begin{lem}\label{iC}
    If we identify two vertices of an immersion $C_i$ and fold, we obtain an immersion $C_k$ for some $k<i$.
\end{lem}
\begin{proof}
We proceed by induction on $i$.
The base case is clear, so let us assume that it holds for $C_l$ with $l<i$.
Notice that the $a$'s form an oriented cycle that traverses every vertex of $C_i$.
Since we are folding afterwards, we can assume that we are identifying $v_0$ with some $v_j$ with $0<j<i$. Let $X \looparrowright K_P$ be the immersion obtained after identifying $v_0$ with $v_j$ and folding. Then $X$ admits an immersion from $C_j$, and therefore it has a quotient of $C_j$ as a subcomplex,  which by induction is some $C_k$. Since $X \looparrowright K_P$ is an immersion and $X$ is connected, by Remark \ref{2saturated} $X$ is equal to $C_k$. 
\end{proof}

\begin{cor}\label{rigid}
    Let $f:X\looparrowright K_P$ be an immersion. Suppose that there exists an immersion $g:C_i\looparrowright K_P$ that factors through $f$ (i.e., there exists an immersion $h:C_i\looparrowright X$ such that $g=f\circ h$). Then $X=C_k$ for some $k$.
\end{cor}
\begin{proof}
    It follows from Lemma \ref{iC} and its proof.
\end{proof}

\begin{lem}\label{bibj}
If we identify $b_i$ to another $b_j$ in $D_i$ and fold, we obtain $C_k$ for some $k$. 
\end{lem}
\begin{proof}
If we identify $b_i$ with $b_0$, we obtain  $C_i$ by definition.
Suppose we identify $b_i$ with $b_j$ with $j\neq 0$.
Then we must identify $v_m$ with $v_n$ for $m\equiv n  \mod (i-j)$.
As a consequence, we identify $b_{i-j}$ with $b_0$.  Let $X \looparrowright K_P$ be the immersion obtained after these identifications and foldings. Then there is an immersion from $C_{i-j}$ to $X$, and by Corollary \ref{rigid}, $X=C_k$ for some $k$.
\end{proof}

\begin{lem}\label{iiD}
If we couple a $2$-cell to $b_i$ in $D_i$, we obtain $D_i$, $D_{i+1}$ or $C_i$.
\end{lem}
\begin{proof}
If the $2$-cell is of type $1$, then we must identify $b_i$ with $b_0$ and we obtain $C_i$.
If the $2$-cell is of type $2$ and we couple it to $b_i$, then depending on which appearance  of $b$ we use, we can get either $D_i$ or $D_{i+1}$.
\end{proof}

Now we can prove the main result of these notes.

\begin{proof}[Proof of Theorem \ref{main}]
Let $X\looparrowright K_P$ be an immersion of a finite, connected $2$-complex without free faces. As we mentioned above, we can assume that it has $2$-cells of both types.
We will construct an immersion of some $C_i$ (or $\tilde{C}_i$)  to $K_P$ that factors through  $X\looparrowright K_P$ as in Corollary \ref{rigid}.
Since $X$ has both types of $2$-cells, it has at least one of type $1$.
We start to build our immersion from this $2$-cell.
Since $X$ has no free faces, the boundary of this $2$-cell must be in the boundary of another $2$-cell of $X$ and, since it is an immersion, the cell must be of type $2$.
We assume that it is attached along the second $b$ (in the other case, we proceed similarly and end up with a $\tilde{C}_i$).
We couple it to the immersion to obtain $D_1$.
Now $D_1$ has a free face labeled by $b_1$. 
Again, since $X$ has no free faces, either this $b_1$ is identified with $b_0$ in $X$, or there is another $2$-cell in $X$ which has it as a face.
If it is identified with $b_0$, we get $C_1$ and we are done.
If it is not, then by Lemma \ref{iiD}, we get $C_1$ or $D_2$.
If we get $D_2$ we continue with this process. Since $X$ is finite and has no free faces, then by Lemmas \ref{bibj} and \ref{iiD}, we must eventually arrive to some $C_i$.
Therefore, by Corollary \ref{rigid} $X$ is a $C_k$, and hence contractible.
\end{proof}

\end{document}